\documentclass[onecolumn,a4paper,12pt]{article}
\usepackage[body={17.4true cm,23true cm}]{geometry}
\usepackage{amsmath,amssymb,amsthm,dsfont,mathrsfs}
\usepackage{indentfirst}
\usepackage{ifpdf}
\usepackage[T1]{fontenc}
\usepackage{indentfirst}
\usepackage{mathpazo}
\usepackage{cite}
\ifpdf
  \usepackage[pdftex]{graphicx}
   \usepackage[pdftex,bookmarks=true,bookmarksnumbered=true,bookmarksopen=true,
     colorlinks=true,citecolor=blue,linkcolor=red,anchorcolor=green,urlcolor=blue]{hyperref}
   \DeclareGraphicsExtensions{.pdf}
\else
   \usepackage[dvips]{graphicx}
   \usepackage[dvipdfm,bookmarks=true,bookmarksnumbered=true,bookmarksopen=true,
     colorlinks=true,citecolor=blue,linkcolor=red,anchorcolor=green,urlcolor=blue]{hyperref}
   \DeclareGraphicsExtensions{.eps}
\fi

\allowdisplaybreaks

\newcommand{\real}{\mathbb{R}}

\newtheorem{lemma}{Lemma}
\newtheorem{theorem}{Theorem}

\newtheorem{definition}{Definition}
\newtheorem{remark}{Remark}

\newenvironment{keywords}{\begin{center}{\bf Keywords}\end{center}\vspace{0em}}{}

\numberwithin{equation}{section}

\title{Synchronized output regulation of nonlinear multi-agent systems
\thanks{This research was supported a grant from the National 863 Program of China (2011AA050204), the National Natural Science Foundation of China (61074122, 61104149), the Zhejiang Province Natural Science Fund (Y1090339,  LY13F030001), the Program for New Century Excellent Talents in University (NCET-11-0459) and the Fundamental Research Funds for the Central Universities (2011QNA4010).}
}
\author{Ji~Xiang$^a$\thanks{Corresponding author. Email: jxiang@zju.edu.cn},~Yanjun~Li$^b$ and Wei Wei$^a$\\
\small $^a$ Department of System Science and Engineering, College of Electrical Engineering, \\[-0.5em]
\small      Zhejiang University, Hangzhou, 310027, China;\\
\small $^b$ School of Information and Electrical Engineering, Zhejiang University City College,\\[-0.5em]
\small      Hangzhou, 310030, China.
}
\date{}
\begin{document}
 \maketitle
\begin{abstract}
This paper considers the synchronized output regulation (SOR) problem of nonlinear
multi-agent systems with switching graph. The SOR means that all agents regulate their
outputs to synchronize on the output of a predefined common exosystem. Each agent
constructs its local exosystem with the same dynamics as that of the common exosystem
and exchanges the state information of the local exosystem. It is shown that the SOR is
solvable under the assumptions same as that for nonlinear output regulation of a single
agent, if the switching graph satisfies the bounded interconnectivity times condition.
Both state feedback and output feedback are addressed. A numerical simulation is made
to show the efficacy of the analytic results.
\end{abstract}

\begin{center}
\begin{keywords}
Synchronized output regulation (SOR), nonlinear system, multi-agent, switching graph.
\end{keywords}
\end{center}

\section{Introduction}
Recent years have witnessed the growing interest in the synchronization of networked systems because it is a ubiquitous phenomena in nature and because of its potential applications on secure communication, distributed generation of the grid, clock synchronization, formation control of multiple robots, and so on. The state synchronization problem might be rooted in the work of Wu and Chua \cite{WuTCASI1995} and recently has been rejuvenated in linear systems with attentions on the accessibility of partial states or switching graph
\cite{ScardoviAU2009} \cite{SeoAU2009} \cite{LiTCASI2010} \cite{DzhunusovARC2011}. Different from the state synchronization that happens between identical systems, the output synchronization can arise between non-identical systems and thereby is more realistic.

Output synchronization for nonlinear input-output passive systems has been studied in \cite{ChopraBK2006}, where under the passive-based design, the output synchronization can be achieved for many cases including balanced graph, nonlinear coupling function and communication delay. In \cite{XiaoSCL2010}, the velocity synchronization problem for second-order integrators has been investigated. The above two studies only take aim at driving the outputs of the agents to each other asymptotically but do not care what the outputs will synchronize on. In \cite{XiangTAC2009}, the linear SOR has been addressed for identical multi-agent systems under the dynamic relative state feedback. There the agents have not only their outputs synchronize but also evolve ultimately on an a trajectory produced by a predefined reference exosystem. In \cite{WielandAU2011}, it is shown that the internal model principle is the sufficient and necessary condition for non-trivial linear output synchronization. There a dynamic controller has been presented for leaderless SOR of linear systems with switching graph. Robust linear SOR have been studied in \cite{WangTAC2010} \cite{KimTAC2011} by only using relative output information. The leader-following SOR of linear systems with switching graph has been investigated in \cite{SuTSMCB2012}.

As for SOR of nonlinear multi-agent systems, there are a few works reported. Gazi \cite{GaziIJC2005} has utilized the nonlinear output regulation method to deal with the formation control problem. There, however, the reference signal, which is stricter than the reference system, is assumed to be known by all agents so that the problem reduces to the completely decoupled output regulation problem. Liu \cite{LiuIETCTA2012} has studied the leader-following SOR under the error feedback for a no-loop graph. Moreover, the robustness is addressed with two extra assumptions: the reference exosystem is linear and the solution of regulator equation are the k-th  polynomials. Xu and Hong \cite{XuCCC2012} have studied the multi-agent systems consisting of two level networks, physical coupling and communication graph. There a networked internal model is proposed for the solution of SOR. They also assume that the graph contain no-loop.

This paper addresses the SOR problem for general nonlinear multi-agent systems with switching topology. Our framework is similar to that in \cite{WielandAU2011} and \cite{SuTSMCB2012} in that the information delivered among the network is assumed to be the state of the local exosystem constructed by agent itself. Both the dynamic state feedback controller and the dynamic output feedback controller are proposed. We show that the SOR can be achieved without extra conditions imposed on the agent dynamics when the switching graph satisfies the bounded interconnectivity times condition (jointly connected condition). The most relevant to our work is the recent work in \cite{LiuND2012} where, however, the graph is fixed and the regulator equation is strengthened to one for the whole multi-agent system so that the result obtained is not scalable.

The remainder of this paper is organized as follows. Problem formulation, as well as
two kinds of controllers, is presented in Section \ref{sec02}. Main results are shown in Section \ref{sec03}; the
exponential synchronization of coupled exosystems is first shown and then synchronized
output regulation is proved for both kinds of controllers. The extension to leader-following
case is addressed in Section \ref{sec04}. A simulation example is illustrated in Section \ref{sec05}, followed by
a conclusion in Section \ref{sec06}.

\section{Problem Statement} \label{sec02}
\subsection{Model}
Consider a multi-agent system consisting of $N$ agents. Each agent has the following dynamics modeled by
\begin{equation}   \label{II01}
\left\{\begin{split}
\dot{x}_i &= f_i(x_i)+g_i(x_i)u_i  \\
y_i & = h_i(x_i),
\end{split}\right., \quad  i=1,2,\cdots,N,
\end{equation}
where $x_i$ is the state, defined on a neighborhood $X_i$ of the origin of $\real^{n_i}$, $u_i\in\real^{m_i}$ is the input, and $y_i\in\real^p$ is the output. The vector $f_i(x_i)$ and the $m_i$ columns of matrix $g_i(x_i)$ are smooth (i.e., $C^{\infty}$) vector fields on $X_i$. $h_i(x_i)$ is a smooth mapping defined on $X_i$.

Each agent drives its output to track the output of a common exosystem, as formulated by
\begin{subequations}
\begin{align}
 &\dot{w}_0=s(w_0)\\
 &y_i+q(w_0)\rightarrow 0.
 \end{align}
\end{subequations}
The first equation describes an autonomous system, the so-called exosystem, defined in a neighborhood $W_0$ of the origin of $\real^s$. The second equation means that the output should track a reference signal produced by the exosystem. The vector $s(w_0)$ is a smooth vector field on $W_0$ and $q(w_0)$ is a smooth map defined on $W_0$.

As for the SOR, the requirements on output $y_i$ are two folds: one is that $y_i$ belongs to a fixed family of trajectories determined by the pair of $(s(w_0), q(w_0))$ with the corresponding initial condition $w_0(0)$ being allowed to vary on a predefined set $\hat{W}_0\subset W_0$; the other is that $y_i\rightarrow y_j$ for all $i,j\in\mathcal{V}$.  The first one is the output regulation problem, which might be solved by a decentralized way and the second one is the synchronization problem, which has to rely on the information exchange to solve.

Generally, a digraph $\mathcal{G}=\{\mathcal{V},\mathcal{E}\}$ is used to depict the communication channels of multi-agent system \eqref{II01}, where node set $\mathcal{V}=\{1,2,\cdots,N\}$ is the index set of agents and edge set $\mathcal{E}\subseteq \{\mathcal{V}\times \mathcal{V}\}$ consists of ordered pair of nodes $(i,j)$, called edge. An edge $(i,j)\in\mathcal{E}$ if and only if there is communication channel from node $i$ to node $j$, where node $i$ is called parent node and node $j$ is called child node. A directed path of digraph is a sequence of edges with form $(i_1,i_2), (i_2,i_3), \cdots$. A tree $\mathcal{G}_t=\{\mathcal{V}_t,\mathcal{E}_t\}$ is a graph where every node has exactly one parent node except for one node, the so-called root node, which has no parent node but has a directed path to every other node. The graph $\mathcal{G}_s=\{\mathcal{V}_s,\mathcal{E}_s\}$ is a subgraph of $\mathcal{G}$ if $\mathcal{V}_s\subseteq\mathcal{V}$ and $\mathcal{E}_s\subseteq \mathcal{E}\bigcap(\mathcal{V}_s\times \mathcal{V}_s)$. The tree $\mathcal{G}_t$ is a spanning tree of graph $\mathcal{G}$ if $\mathcal{G}_t$ is a subgraph of $\mathcal{G}$ with $\mathcal{V}_t=\mathcal{V}$.

A switching graph, defined on a piecewise constant switching signal $\sigma(t):\real^+\mapsto \mathcal{P}=\{1,2,\cdots,P\}$, is denoted by $\mathcal{G}_{\sigma(t)}=\{\mathcal{V},\mathcal{E}_{\sigma(t)}\}$, where set $\mathcal{P}$ indexes the total $P$ number digraphs and $\mathcal{E}_{s}\subseteq\mathcal{V}\times \mathcal{V}$ with $s\in\mathcal{P}$. The time instants when $\sigma$ switches is denoted by an increasing sequence $t_k$, $k=0,1,2,\cdots,$ with $t_0=0$. Denote by $\sigma_k$ the value of $\sigma(t)$ when $t\in[t_{k-1},t_k)$. Denote by $\mathcal{A}^{\sigma_k}$ and $L^{\sigma_k}$ the adjacency matrix and the Laplacian matrix of $\mathcal{G}_{\sigma_k}$, respectively. We assume that any two consecutive switching instants are separated by a dwell-time $D_t$, i.e., $t_k-t_{k-1}\ge D_t$ so as to guarantee that the switching graph is non-chattering and zeno behavior cannot occur. A union graph $\mathcal{G}_{[t^1,t^2]}$ over an interval time $[t^1, t^2]$ is defined by $\mathcal{G}_{[t^1,t^2]}\triangleq (\mathcal{V}, \bigcup_{t\in[t^1,t^2]}\mathcal{E}_{\sigma(t)})$ that corresponds to a graph consisting of all nodes in $\mathcal{V}$ and all edges that appear at any time $t\in[t^1,t^2]$.

In order for the SOR, one natural route is firstly to synchronize the exosystems of all agents by exchanging the their state and then to drive the agent output to track the output of the local exosystem \cite{WielandAU2011}\cite{SuTSMCB2012}. In the first step, each agent builds the following coupled exosystem, based on the communication graph $\mathcal{G}_{\sigma(t)}$,
\begin{equation} \label{II02}
  \dot{w}_i = s(w_i)+\sum_{j=1}^N a_{ij}^{\sigma(t)}(w_j-w_i),
\end{equation}
where $a_{ij}^{\sigma(t)}$ denotes the $i$th row and $j$th column element of adjacency matrix $\mathcal{A}^{\sigma(t)}$ of graph $\mathcal{G}_{\sigma(t)}$. If $(j,i)\in\mathcal{E}_{\sigma(t)}$, then $a^{\sigma(t)}_{ij}>0$; otherwise, $a^{\sigma(t)}_{ij}=0$. In this case,  the tracking error for each agent is defined as
\begin{equation}
  \label{II03}
  e_i = h_i(x_i)+q(w_i).
\end{equation}

\subsection{Controllers}
Two kinds of controllers are considered in this paper,
\begin{itemize}
  \item [1)] \textbf{Distributed dynamic state feedback controller}
\begin{equation} \label{II04}
  \left\{
  \begin{split}
    & \dot{w}_i = s(w_i)+\sum_{j=1}^N a_{ij}^{\sigma(t)}(w_j-w_i),\\
    & u_i = \alpha_i(x_i,w_i)
  \end{split}
  \right.,  \quad i\in\mathcal{V},
\end{equation}
where $\alpha_i(x_i,w_i)$ is a $C^k$ (for some integer $k\ge2$) mapping defined on $X_i\times W$, satisfying $a_i(0,0)=0$. Combining \eqref{II04} and \eqref{II01} yields the following closed-loop system,
\begin{equation}
  \label{II05}
  \left\{\begin{split}
    \dot{x}_i& = f_i(x_i)+g_i(x_i)\alpha_i(x_i,w_i)  \\
    \dot{w}_i& = s(w_i)+\sum_{j=1}^N a_{ij}^{\sigma(t)}(w_j-w_i)
  \end{split}, \right.\quad i\in\mathcal{V},
\end{equation}
which has an equilibrium at, $(x_i,w_i)=(0,0)$ for all $i\in\mathcal{V}$.
\item [2)] \textbf{Distributed dynamic output feedback controller}
\begin{equation}
\label{II06}
\left\{ \begin{split}
\dot{w}_i& = s(w_i)+\sum_{j=1}^N a_{ij}^{\sigma(t)}(w_j-w_i)\\
  \dot{z}_i& = \eta_i(z_i,y_i)\\
  u_i &= \alpha_i(z_i,w_i)
\end{split},
\right. \quad i\in\mathcal{V},
\end{equation}
where $z_i\in\real^{n_i}$ is the observer state, defined on a neighborhood $Z_i$ of the origin of $\real^{n_i}$.  For each $y_i\in\real^p$, $\eta_i(z_i,y_i)$ is a $C^k$ vector field on $Z_i$ (for some integer $k\ge 2$). The closed-loop system under controller \eqref{II06} has the form
\begin{equation} \label{II07}
  \left\{
  \begin{split}
    \dot{x}_i&=f_i(x_i)+g_i(x_i)\alpha_i(z_i,w_i)\\
    \dot{z}_i &  = \eta_i(z_i,h_i(x_i))\\
    \dot{w}_i & = s(w_i) + \sum_{j=1}^N a_{ij}^{\sigma(t)}(w_j-w_i)
  \end{split},
  \right.\quad i\in\mathcal{V},
\end{equation}
which has an equilibrium, $(x_i,z_i,w_i)=(0,0,0)$ for all $i\in\mathcal{V}$, when $\eta_i(0,0)=0$ for all $i\in\mathcal{V}$.
\end{itemize}

The purpose of SOR includes three aspects, local asymptotically stable, output regulation and output synchronization. Define the stacked vector $x=[x_1^T,\cdots,x_N^T]^T$, $w=[w_1^T,\cdots,w_N^T]^T$ and $z=[z_1^T,\cdots,z_N^T]^T$. Correspondingly, their domains are defined as $X=X_1\times\cdots\times X_N$, $W=W_1\times\cdots\times W_N$, and $Z=Z_1\times\cdots\times Z_N$. Formally, the following two problems are proposed,
\begin{definition}
  [\textbf{State Feedback Synchronized Regulator Problem}] Find, if possible, $\alpha_i(x_i,w_i)$ for node $i$ such that:
  \begin{enumerate}
    \item [1a)] the equilibrium $x=0$ of
    \begin{equation}  \label{II08}
      \dot{x}_i = f_i(x_i)+g_i(x_i)\alpha_i(x_i,0), \quad i\in\mathcal{V}
    \end{equation}
    is exponentially stable.
    \item [1b)] there are a neighborhood $U\subset X\times W$ of $(0,0)$ and a exosystem,
    \begin{equation} \label{II09}
      \dot{w}_0=s(w_0),
    \end{equation}
    defined on the neighborhood $W_0$ of origin of $\real^{s}$, such that, for each initial condition $(x(0),w(0))\in U$, there exists a initial condition $w_0(0)\in W_0$ for system \eqref{II09} such that the solution of \eqref{II05} satisfies
    \begin{equation}
      \label{II10}
      \lim_{t\rightarrow \infty} (h_i(x_i(t))+q(w_0(t))) = 0, \quad \forall\, i\in\mathcal{V}.
    \end{equation}
  \end{enumerate}
\end{definition}

\begin{definition}
  [\textbf{Error feedback Synchronized Regulator Problem}] Find, if possible, $\alpha_i(z_i,w_i)$ and $\eta_i(z_i,y_i)$, such that
  \begin{itemize}
    \item [2a)] the equilibrium $(x,z)=0$ of
    \begin{equation} \label{II11}
    \begin{split}
      &\dot{x}_i=f_i(x_i)+g_i(x_i)\alpha_i(z_i,0),\\
      & \dot{z}_i=\eta_i(z_i,h_i(x_i)),
      \end{split}\quad \quad i\in\mathcal{V},
    \end{equation}
  is exponentially stable.
  \item [2b)] there exists a neighborhood $U\subset X\times Z\times W$ of $(0,0,0)$ and an exosystem
  \begin{equation} \label{II12}
    \dot{w}_0=s(w_0)
  \end{equation}
  defined on a neighborhood $W_0$ of origin of $\real^{s}$, such that, for each initial condition $(x(0), z(0)$, $ w(0))\in U$, there exists a initial condition $w_0(0)\in W_0$ for system \eqref{II12} such that the solution of \eqref{II07} satisfies
  \begin{equation}
    \label{II13}
    \lim_{t\rightarrow \infty} (h_i(x_i(t))+q(w_0(t)))=0, \quad \forall\, i\in\mathcal{V}.
  \end{equation}
  \end{itemize}
\end{definition}

\medskip
For the solvability of the above problems, the following assumptions are made,
\begin{enumerate}
  \item [A1)] The exosystem modeled by $\dot{w}_0=s(w_0)$ has a stable equilibrium at $w_0=0$, and there is an open neighborhood of $w_0=0$ in which every point is Poisson stable.
  \item [A2)] The pair $\big(f_i(x_i),g_i(x_i)\big)$ has a stabilizable linear approximation at $x_i=0$, for all $i\in\mathcal{V}$.
  \item [A3)] The pair $\big(f_i(x_i),h_i(x_i)\big)$ has a detectable linear approximation at $x_i=0$ for all $i\in\mathcal{V}$.
  \item [A4)] There is a bounded time length $T$ and a starting time $t>0$ such that for each $k$, $k=1,2,\cdots$, the union graph $\mathcal{G}_{[t+(k-1)T,t+kT]}$ has a spanning tree embedded.
\end{enumerate}

Assumptions A1)$\sim$A3) are standard for nonlinear output regulation problems \cite{IsidoriTAC1990}. Assumption A4) is referred to as the bounded interconnectity times condition in \cite{XiangIJC2011}, which is the weakest condition for the consensus seeking of a switching diagraph and has many invariant versions, such as the jointly connected condition \cite{HongTAC2007} and uniformly quasi-strongly connected condition \cite{lin2006Thesis}.

\section{Main results} \label{sec03}
\subsection{Exponentially synchronization of coupled exosystems}
Noticing that the coupled exosystem \eqref{II02} is independent of the agent dynamics, the closed-loop systems \eqref{II05} and \eqref{II07} can be regarded as to be driven by an lumped exosystem $w$ of $Ns$ dimensions. Since the lumped exosystem has the dimension in excess of what is required, its dynamics must contain some decay modes, that is, in some vector directions $w$ is asymptotically converging to zeros. In order for the second condition 1b) or 2b), the undecayed mode must be the flow determined by the vector field $s(w_0)$, which means that all the exosystems should synchronize. To this end, the following result is recalled (Corollary 7 in \cite{XiangIJC2011}),
and rephrased as follows,
\begin{lemma} \label{le01}
Given a multi-agent system \eqref{II02} with communication graph $\mathcal{G}_{\sigma(t)}$satisfying assumption A4). If the largest Lyapunov exponent $\nu_{max}$ of system $\dot{w}_0=s(w_0)$ and the consensus convergence rate $\alpha^*(T,t)$ of $\mathcal{G}_{\sigma(t)}$ are such that
\begin{equation}
  \label{III01}
  \nu_{max}+\ln(\alpha^*(T,t))/T<0
\end{equation}
then system \eqref{II02} is locally exponentially synchronizable.
\end{lemma}

Denote by $\Phi_t^s(w^0)$ the flow of vector field $s(w_0)$, defined for all $t\in\real$,  with $w_0(0)=w^0$. Then the maximum Lyapunov exponent of dynamic system $\dot{w}_0=s(w_0)$ is defined as \cite{CenciniBOOK2010}
\begin{equation}
  \nu_{max} = \lim_{t\rightarrow \infty} \lim_{\delta w^0\rightarrow 0}\frac{1}{t} \ln\frac{\|\Phi_t^s(w^0+\delta w^0)-\Phi_t^s(w^0)\|}{\|\delta w^0\|}
\end{equation}

On the other hand, consider the consensus rate of a switching graph. Define the time interval sequence $\{\delta_k\}$ by $\delta_k = t_k-t_{k-1}$, $k=1,2,\cdots$. Let $t^1$ and $t^2$ be two time instants located in the time slots of $[t_{i-1}, t_i)$ and $[t_{j-1},t_j)$ with $i\le j$, respectively.

Given a switching graph $\mathcal{G}_{\sigma(t)}$ satisfying Assumption A4) with given $t$ and $T$, its
consensus convergent rate is defined as the supremum of the contract rate of transition matrix $\Phi(t+(k-1)T,t+kT)$ \cite{XiangIJC2011},
\begin{equation}
  \alpha^*(T,t)=\sup_k \alpha(\Phi(t+(k-1)T,t+kT))
\end{equation}
where the contract rate $\alpha(\Phi)$ is defined as $\alpha(\Phi) = \max\limits_{x\neq 0, x\bot \mathbf{b}}\frac{\|\Phi^Tx\|}{\|x\|}$, and the transition matrix is defined as
\begin{equation} \label{III03}
  \Phi(t^2,t^1)=e^{-L^{\sigma_j}(t^2-t_{j-1})}e^{-L^{\sigma_{j-1}}\delta_{j-1}}\cdots e^{-L^{\sigma_{i+1}}\delta_{i+1}}e^{-L^{\sigma_{i}}(t_i-t^1)}
\end{equation}
where $L^{\sigma_k}$ is the Laplacian matrix of graph $\mathcal{G}_{\sigma_k}$. Throughout of this paper, $\mathbf{b}$ denotes the vector with all elements being $1$.

\begin{lemma} \label{le02}
Given a coupled exosystem \eqref{II02} with assumptions A1) and A4), then there is a dynamic system $\dot{w}_0=s(w_0)$ such that for all initial conditions $w(0)\in\mathcal{\bar{W}}\subset W$, there is a initial condition $w_0(0)\in W_0$ such that the error $\delta w_i=w_i-w_0$  exponentially converges to zero, for all $i\in\mathcal{V}$.
\end{lemma}
\begin{proof}
According to Assumption A4) and definition \eqref{III03}, it follows that transition matrix $\Phi(t+(k-1)T,t+kT)$ is stochastic, indecomposable and aperiodic, and has exactly one trivial eigenvalue $1$ associated with eigenvector $\mathbf{b}$. Therefore, $\alpha^*(T,t)<1$.

On the other hand, with assumption A1), $\nu_{max}\le 0$. If not, for any given $\epsilon$, there is a time $T_\epsilon$, such that for any $0<\|\delta w^0\|<\epsilon$, $\|\Phi_t^s(w^0+\delta w^0)\|>\epsilon$ for all $t>T_\epsilon$. This is contradictory to the feature of Poisson stable \cite{IsidoriTAC1990}.

With $\nu_{max}\le0$ and $\alpha^*(T,t)<1$, making use of Lemma \ref{le01} yields that there is a neighborhood $\bar{W}\subset W$ of origin, for all initial condition $w(0)\in\bar{W}$, all the exosystems have their states exponentially synchronize on a manifold determined by $\dot{w}_0=s(w_0)$ with $w_0\in\mathcal{W}^0$.
\end{proof}

\subsection{Solution of static feedback synchronized regulator problem}
Before proceeding the main results, some matrices are firstly introduced, arising from the linearization of nonlinear dynamics on the equilibrium of origin.
\begin{equation}
  A_i = \left[ \frac{\partial f_i}{\partial x_i}\right]_{x_i=0},  \quad B_i = g_i(0), \quad \quad C_i= \left[ \frac{\partial h_i }{\partial x_i}\right]_{x_i=0}.
  \end{equation}

Denote by $\psi_{f}$ the reminder of the linear approximation of a vector function $f$, that is,
\begin{equation}
  f(x)=f(0)+\left.\frac{\partial f}{\partial x}\right|_{x=0} x + \psi_f(x).
\end{equation}
Let $\mathrm{D}f$ denote the Jacobian matrix of vector function $f$, that is, $\mathrm{D}f(x) = \frac{\partial f}{\partial x}$. It follows that $\lim_{x\rightarrow 0}\psi_f(x)=0$ and $\lim_{x\rightarrow 0}\mathrm{D}\psi_f(x)=0$.
Define
\begin{equation} \label{III07}
 \mathrm{ M}f(x,h)=\left(\int^1_0 \mathrm{D}f(x+th)dt\right),
\end{equation}
then by mean-value theorem,
\begin{equation} \label{III08}
\psi_f(x+h)=\psi_f(x)+\mathrm{M}\psi_f(x,h)h.
\end{equation}

Secondly,a useful lemma is presented, which plays a key role in the proof of main results for both state feedback and output feedback cases.
\begin{lemma}
  \label{le03}
Given a multi-agent system \eqref{II01} with assumptions A1), A2) and A4). Suppose that for all $i\in\mathcal{V}$, there exist $C^k(k\ge 2)$ mapping $x_i=\pi_i(w_0)$, with $\pi_i(0)=0$, and $u_i=c_i(w_0)$, with $c_i(0)=0$, both defined in a neighborhood $W^0$ of origin, satisfying the conditions
\begin{subequations} \label{III09}
\begin{align}
 \label{III09a} &\frac{\partial \pi_i}{\partial w_0} s(w_0) = f_i(\pi_i(w_0))+g_i(\pi_i(w_0))c_i(w_0) \\
 \label{III09b} &h_i(\pi_i(w_0))+q(w_0)=0
\end{align}
\end{subequations}
then under the following controller
\begin{equation}
  \label{III10}
  u_i = \alpha_i(x_i,w_i)=c_i(w_i)+K_i(x_i-\pi_i(w_i)), \quad i\in\mathcal{V}.
\end{equation}
where $K_i$ is such that $A_i+B_iK_i$ is Hurwitz, conditions 1a) and 1b) will be satisfied.
\end{lemma}
\begin{proof}
According to Assumption A2), it is true that there exists a matrix $K_i$ such that $A_i+B_iK_i$ is Hurwitz for all $i\in\mathcal{V}$. Noting that $\alpha_i(x_i,0)=K_ix_i$, condition 1a) follows directly. Below we show condition 1b).

By Lemma \ref{le02}, there is a positive scalar $\beta_1$ such that
\begin{equation} \label{III11}
  \|\delta w_i(t)\| \le \epsilon_0 e^{-\beta_1t}\|\delta w_i(0)\|, \quad \forall\,\, i\in\mathcal{V},
\end{equation}
for some positive scalar $\epsilon_0$. Consider the vector $e_{xi}=x_i-\pi_i(w_0)$. With \eqref{III09a}, its dynamics has the form
\begin{equation} \label{III12}
  \dot{e}_{xi}=f_i(x_i)+g_i(x_i)\alpha_i(x_i,w_i)-f_i(\pi_i(w_0))-g_i(\pi_i(w_0))c_i(w_0).
\end{equation}
Making use of linear approximation and mean value theorem, one has,
\begin{equation}
  f_i(x_i)-f_i(\pi_i(w_0))=A_ie_{xi}+\psi_{f_i}(x_i)-\psi_{f_i}(\pi_i(w_0)) = A_ie_{xi}+\mathrm{M}\psi_{f_i}(\pi_i(w_0),e_{x_i})e_{x_i}
\end{equation}
and
\begin{equation}
\begin{split}
 g_i(x_i)&\alpha_i(x_i,w_i)-g_i(\pi_i(w_0))c_i(w_0)
 = g_i(x_i)c_i(w_i)-g_i(\pi_i(w_0))c_i(w_0)+g_i(x_i)K_ie_{xi}\\
 =&(g_i(x_i)-g_i(\pi_i(w_0)))c_i(w_0)+g_i(x_i)(\mathrm{M}c_i(w_0,\delta{w_i})\delta{w_i}+K_ie_{xi})\\
 =&B_iK_ie_{xi}+
  \left[\mathrm{M}g_{ij}(0,x_i)x_i,\cdots, \mathrm{M}g_{im_i}(0,x_i)x_i\right]K_ie_{xi}+ g_i(x_i)\mathrm{M}c_i(w_0,\delta{w_i})\delta{w_i}
 \\& +\left(\sum_{j=1}^{m_i} c_{ij}(w_0)\mathrm{M}g_{ij}(\pi_i(w_0),e_{xi})\right)e_{xi}
 \end{split}
\end{equation}
where $g_{ij}(x_i)$ denotes the $j$th column vector of matrix function $g_i(x_i)$ and $c_{ij}(w_0)$ denotes the $j$th element of $c_i(w_0)$. With the above two equations, equation \eqref{III12} can be rewritten as
\begin{equation}
  \dot{e}_{xi}=(A_i+B_iK_i)e_{xi}+ N(w_0,x_i)e_{xi}+g_i(x_i)\mathrm{M}c_i(w_0,\delta w_i)\delta w_i, \quad i\in\mathcal{V},
\end{equation}
where
\begin{equation}
\begin{split}
N(w_0,x_i)&=\mathrm{M}\psi_{f_i}(\pi_i(w_0),x_i-\pi_i(w_0))+\sum_{j=1}^{m_i} c_{ij}(w_0)\mathrm{M}g_{ij}(\pi_i(w_0),x_i-\pi_i(w_0)) \\&+
  \left[\mathrm{M}g_{ij}(0,x_i)x_i,\cdots, \mathrm{M}g_{im_i}(0,x_i)x_i\right]K_i
\end{split}.
\end{equation}
Since $A_i+B_iK_i$ is Hurwitz, there are a symmetric positive definite matrix $P_i\in\real^{n_i\times n_i}$ and a positive scalara $\epsilon_1$ such that
\begin{equation}
  \label{III17}
P_i(A_i+B_iK_i)+(A_i+B_iK_i)^TP_i+2\epsilon_1P_i<-I_{n_i}.
\end{equation}
By assumption A1) and noticing that condition 1a) holds, there exist sufficiently small $w_i(0)$ and $x_i(0)$ (notice that $w_0$ depends on $w_i$, $i\in\mathcal{V}$), such that the trajectories of $x_i(t)$ and $w_i(t)$ of the closed-loop system \eqref{II05} satisfy
\begin{equation}
  \|N(w_0(t),x_i(t))\|\le \frac{1}{2\lambda_M(P_i)},\,\, \,\, \|g_i(x_i(t))\mathrm{M}c_i(w_0(t),\delta w_i(t))\|\le \frac{1}{\lambda_M(P_i)}, \quad \forall t>0.
\end{equation}
where $\lambda_M(P_i)$ denotes the maximum eigenvalue of $P_i$.

Then consider the Lyapunov function $V_i=e_{xi}^TP_ie_{xi}$, whose derivative satisfies
\begin{equation}
  \dot{V}_i\le - 2\epsilon_1e_{xi}^TP_ie_{xi} + 2e_{xi}^TP_i g_i(x_i)\mathrm{M}c_i(w_0,\delta w_i)\delta w_i
  \le -2\epsilon_1 V_i + 2\epsilon_2 \sqrt{V_i}\|\delta w_i\|,
\end{equation}
where $\epsilon_2>\frac{1}{\sqrt{\lambda_{m}(P_i)}}$ is a constant scalar and $\lambda_m(P_i)$ denoting the minimal eigenvalue of $P_i$.

Define $\bar{V}_i=\sqrt{V}_i$, then $\bar{V}_i\ge0$ and
\begin{equation}
  \dot{\bar{V}}_i\le -\epsilon_1\bar{V}_i+\epsilon_2 \|\delta w_i\|,
\end{equation}
which is equivalent to $\left(e^{\epsilon_1t}\bar{V}_i\right)'\le \epsilon_2 e^{\epsilon_1t}\|\delta w_i\|$. With \eqref{III11}, one further obtains
\begin{equation}
  \bar{V}_i(t)\le  e^{-\epsilon_1 t}\left(\bar{V}_i(0)-\frac{\epsilon_2\epsilon_0\|\delta w_i(0)\|}{\epsilon_1-\beta_1} \right) + e^{-\beta_1 t}\frac{\epsilon_2\epsilon_0\|\delta w_i(0)\|}{\epsilon_1-\beta_1},
\end{equation}
from which, $\bar{V}_i(t)\rightarrow 0$ as $t\rightarrow \infty$, and so does $e_{xi}$. Using \eqref{III09b}, it can be further concluded that condition 1b) will be satisfied.
\end{proof}

\begin{remark}
It should be pointed out that the proof of the above lemma is not based on the center manifold method, which can not be directly applied here in the presence of switching graphs.
\end{remark}

\bigskip
Now we are ready to present the main result for state feedback case.
\begin{theorem}
  \label{th01}
Under assumption A1), A2) and A4), the state feedback synchronized regulator problem is solvable for the multi-agent system \eqref{II01} if and only if for all $i\in\mathcal{V}$, there exist $C^k(k\ge 2)$ mapping $x_i=\pi_i(w_0)$, with $\pi_i(0)=0$, and $u_i=c_i(w_0)$, with $c_i(0)=0$, both defined in a neighborhood $W^0$ of origin, satisfying the conditions
\begin{subequations} \label{III09}
\begin{align}
 \label{III09a} &\frac{\partial \pi_i}{\partial w_0} s(w_0) = f_i(\pi_i(w_0))+g_i(\pi_i(w_0))c_i(w_0) \\
 \label{III09b} &h_i(\pi_i(w_0))+q(w_0)=0
\end{align}
\end{subequations}
\end{theorem}
\begin{proof}
Necessity is obvious by considering the special situation with $N=1$.  The sufficiency follows immediately from Lemma \ref{le03}.
\end{proof}

\begin{remark}
Theorem \ref{th01} says that the solvability condition for the local SOR is the same as that for the local output regulation of a single agent.
\end{remark}

\subsection{Solution of output feedback synchronized regulator problem}
\begin{theorem}\label{th02}
Under assumptions A1)$\sim$A4), the output feedback synchronized regulator problem is solvable for the multi-agent system \eqref{II01} if and only if there exist $C^k(k\ge 2)$ mapping $x_i=\pi_i(w_0)$, with $\pi_i(0)=0$, and $u_i=c_i(w_0)$, with $c_i(0)=0$, both defined in a neighborhood $W^0$ of origin, satisfying the conditions
\setcounter{equation}{8}
\begin{subequations} \label{III09}
\begin{align}
 \label{III09a} &\frac{\partial \pi_i}{\partial w_0} s(w_0) = f_i(\pi_i(w_0))+g_i(\pi_i(w_0))c_i(w_0) \\
 \label{III09b} &h_i(\pi_i(w_0))+q(w_0)=0
\end{align}
\end{subequations}
\end{theorem}
\begin{proof}
Necessity is clear. Below we show the sufficiency by using a constructive method. Assumption A2) and A3) mean that there are matrices $K_i$ and $L_i$ such that
\setcounter{equation}{21}
\begin{equation}
  A_i+B_iK_i \quad \mbox{and}\quad A_i+L_iC_i
\end{equation}
are Hurwitz. By them, the following matrix
\begin{equation} \label{III23}
  \begin{bmatrix}
    A_i & B_iK_i \\
    -L_iC_i & A_i+B_iK_i+L_iC_i
  \end{bmatrix}
\end{equation}
is also Hurwitz. Suppose there are two maps $\pi_i(w_0)$ and $c_i(w_0)$ satisfying \eqref{III09}, then set the dynamic controller to be
\begin{subequations} \label{III24}
\begin{align}
&u_i=\alpha_i(z_i,w_i)=c_i(w_i)+K_i(z_i-\pi_i(w_i)), \label{III24a}\\
&\dot{z}_{i}=\eta_i(z_i,h_i)=f_i(z_{i})+g_i(z_{i})u_i+L_i(h_i(z_i)-h_i(x_i)), \label{III24b}
\end{align}
\end{subequations}
for all $i\in\mathcal{V}$. Define the augmented state $\tilde{x}_i = [x_i^T, z_i^T]^T$, and maps
\begin{gather}
  \label{III25}
  \tilde{f}_i({\tilde{x}_i})=\begin{bmatrix}
    f_i(x_i)\\
    -L_ih_i(x_i)+L_ih_i(z_i)+f_i(z_i)
  \end{bmatrix}, \quad \tilde{g}_i(\tilde{x}_i)=\begin{bmatrix}
    g_i(x_i) \\ g_i(z_i)
  \end{bmatrix}, \nonumber \\
  \tilde{\pi}_i(w_0)=\begin{bmatrix}
    \pi_i(w_0) \\ \pi_i(w_0)
  \end{bmatrix}, \quad \tilde{h}_i(\tilde{x}_i)=h_i(x_i).
\end{gather}
With them, one has
\begin{equation} \label{III26}
  \dot{\tilde{x}}_i = \tilde{f}_i(\tilde{x}_i)+\tilde{g}_i(\tilde{x}_i)u_i, \quad \tilde{y}_i=\tilde{h}_i(\tilde{x}_i), \quad i\in\mathcal{V},
\end{equation}
and
\begin{equation} \label{III27}
\begin{split}
 &\frac{\partial \tilde{\pi}_i}{\partial w_0} s(w_0) = \tilde{f}_i(\tilde{\pi}_i(w_0))+\tilde{g}_i(\tilde{\pi}_i(w_0))c_i(w_0), \\
&\tilde{h}_i(\tilde{\pi}_i(w_0))+q(w_0)=0.
\end{split}
\end{equation}
The Jacobian matrices of $\tilde{f}_i(\tilde{x}_i)$ and $\tilde{g}_i(\tilde{x}_i)$ have the form of
\begin{equation}
  \tilde{A}_i=\left[
  \frac{\partial \tilde{f}_i}{\partial \tilde{x}_i}
  \right]_{\tilde{x}_i=0}=\begin{bmatrix}
    A_i & 0 \\
    -L_iC_i & A_i+L_iC_i
  \end{bmatrix}, \quad \tilde{B}_i=\left[
  \frac{\partial \tilde{g}_i}{\partial \tilde{x}_i}
  \right]_{\tilde{x}_i=0}=\begin{bmatrix}
    B_i & 0 \\
    0 & B_i
  \end{bmatrix}
\end{equation}
Noting that $\tilde{A}_i+\tilde{B}_i[0, K_i]$ has exactly the form \eqref{III23} and that controller \eqref{III24a} can be rewritten as $u_i=c_i(w_i)+[0, K_i](\tilde{x}_i-\tilde{\pi}_i(w_i))$, by Lemma \ref{le03} conditions 1a) and 1b) are satisfied for multi-agent system \eqref{III26} with controller \eqref{III24a}, and subsequently conditions 2a) and 2b) are satisfied for the closed-loop system \eqref{II07}. Therefore the output feedback synchronized regulator problem is solved under controller \eqref{III24}.
\end{proof}

\begin{remark}
Again it is without extra requirements for solving the output feedback synchronized regulator problems. This might be understood in the sense that sometimes the local property of a nonlinear system can be obtained from its linearization system. Besides, compared with the solution of state feedback case, an additional observer-like compensator is added here.
\end{remark}

\section{Extension to Leader-following case} \label{sec04}
The contents above are for the leaderless case, which only requires that the agents have their outputs synchronize on the common manifold, but does not designate their values that are self-organized. While many real applications, such as power networks, require the output of each agent synchronizes on a designated reference trajectory. Such a case is referred to as leader-following output regulation and has been studied for linear systems, for example in \cite{XiangTAC2009} \cite{WangTAC2010} and \cite{SuTAC2012}. In this section, we further extend the above results to the leader-following case for nonlinear systems.

Similar to that in \cite{WangTAC2010}, set the reference exosystem
\begin{equation} \label{IV01}
  \dot{w}_0=s(w_0)
\end{equation}
to be leader node, indexed as $0$, so as to form an augmented multi-agent system over an augmented graph $\bar{\mathcal{G}}=\{\bar{\mathcal{V}},\bar{\mathcal{E}}\}$, where $\bar{\mathcal{V}}=\{0\cup\mathcal{V}\}$ and $\bar{\mathcal{E}}=\{\mathcal{E}_0\cup\mathcal{E}\}$. Edge set $\mathcal{E}_0\subseteq \{0 \times \mathcal{V}\}$. $(0,i)\in\mathcal{E}_0$ if and only if node $i$ has the state information of reference exosystem $w_0$. In such a configuration, the first equation in controllers both \eqref{II04} and \eqref{II06} should be changed to be
\begin{equation}
  \dot{w}_i=s(w_i)+\sum_{j=0}^N \bar{a}_{ij}^{\sigma(t)}(w_j-w_i), \quad i\in\mathcal{V},
\end{equation}
where $\bar{a}^{\sigma(t)}_{ij}$ denotes the element of adjacency matrix $\bar{A}^{\sigma(t)}$ of diagraph $\bar{\mathcal{G}}_{\sigma(t)}$. Also, Assumption A4) should be replaced by
\begin{enumerate}
 \item[A5)] There is a bounded time length $T$ and a starting time $t>0$ such that for each $k$, $k=1,2,\cdots$, the union graph $\bar{\mathcal{G}}_{[t+(k-1)T,t+kT]}$ has a spanning tree embedded.
\end{enumerate}
\begin{remark}
 Assumption A4) is not necessary for Assumption A5). On the other hand, noticing that edge $(i,0)$ does not belong to $\bar{\mathcal{E}}$, a spanning tree, if exists, must be rooted at node $0$.
\end{remark}

Below for simplicity, the result of leader-following output regulation only for output feedback case is presented straightforwardly.
\begin{theorem}
Given a multi-agent system \eqref{II01} and an exosystem \eqref{IV01} satisfying assumptions A1)$\sim$A3) and $A5)$. There is a dynamic controller of the form
\begin{equation}
  \left\{ \begin{split}
  \dot{w}_i& =s(w_i)+\sum_{j=0}^N \bar{a}_{ij}^{\sigma(t)}(w_j-w_i))\\
  \dot{z}_i& = \eta_i(z_i,y_i)\\
    u_i &= \alpha_i(z_i,w_i)
\end{split},
\right. \quad i\in\mathcal{V}
\end{equation}
such that for all sufficiently small $x_i(0)$, $w_i(0)$, $z_i(0)$ and $w_0(0)$, the trajectory of the closed-loop system is bounded and satisfies
\begin{equation}
  \lim_{t\rightarrow 0} \left(h_i(x_i)+q(w_0)\right)\rightarrow 0,
\end{equation}
if and only if there exist $C^k(k\ge 2)$ mapping $x_i=\pi_i(w_0)$, with $\pi_i(0)=0$, and $u_i=c_i(w_0)$, with $c_i(0)=0$, both defined in a neighborhood $W^0$ of origin, satisfying the conditions
\begin{equation}
\begin{split}
 &\frac{\partial \pi_i}{\partial w_0} s(w_0) = f_i(\pi_i(w_0))+g_i(\pi_i(w_0))c_i(w_0) \\
&h_i(\pi_i(w_0))+q(w_0)=0
\end{split}, \quad i\in\mathcal{V}.
\end{equation}
\end{theorem}

\begin{remark}
For a leaderless multi-agent system, if one agent does not adjust the exosystem constructed by itself so that equivalently the agent does not receive the information of others (but sent its information to others), and assumption A4) is still satisfied, then the leaderless case reduces to the leader-following case. In this consideration, the leader-following case can be regraded as a special leaderless case.
\end{remark}

\section{Simulation Example} \label{sec05}
For the sack of simpleness, a multi-agent system of three nodes is taken as an illustrated example. These nodes are described respectively by the following equations.
\begin{figure}
  \centering
  \includegraphics[width=0.8\textwidth]{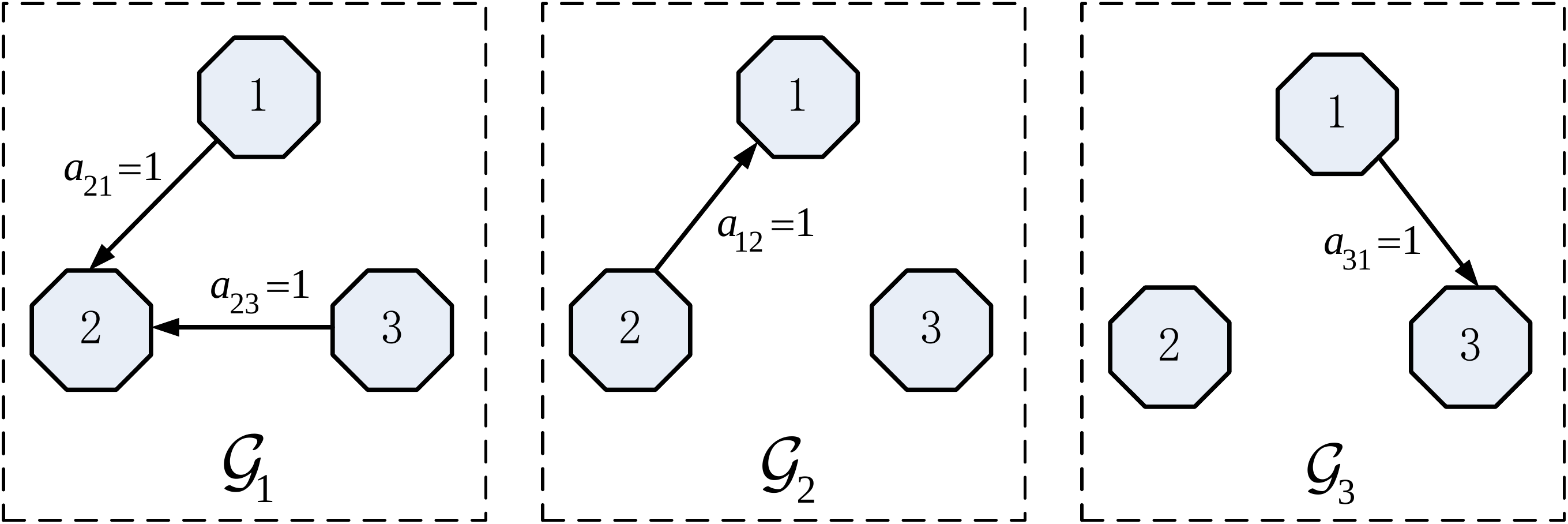}\\
  \caption{Three kinds of graphs involved in the communication of agents}\label{fig01}
\end{figure}

\begin{itemize}
\item Agent-1 is with state $x_1\in\real$, and $\dot{x}_1=x_1^2+u_1$ and $y_1=x_1$.
\item Agent-2 is with state $x_2\in\real^2$ and
\begin{equation}
\begin{split}
  &\dot{x}_{21}= -x_{21}+x_{22}, \quad \dot{x}_{22}=x_{21}^2+u_{22},\\
  &y_{2}=x_{21}
\end{split}
\end{equation}
\item Agent-3 with state $x_3\in\real^2$ and
\begin{equation}
  \begin{split}
    &\dot{x}_{31}=x_{32},\quad  \dot{x}_{32}=-x_{31}+x_{32} - x_{31}^3+u_3 \\
   & y_3 = x_{31}
  \end{split}
\end{equation}
\end{itemize}
The dynamics that their outputs want to manifest is a sinusoid wave, formulated by
\begin{equation}
  \dot{w}_0=s(w_0)=\begin{bmatrix}
    0 & \tau\\
    -\tau & 0
  \end{bmatrix}w_0, \quad q(w_0)=w_{01}
\end{equation}
where $\tau$ denotes the angle frequency of the sinusoid wave. Here notations $x_{ij}$ and $w_{ij}$ denote the $j$-th element of $x_i$ and $w_i$, respectively. It can be verified that for the three agents, the regulator equation \eqref{III09} has solutions with, respectively,
\begin{subequations}
\begin{align}
  &\pi_1(w_0)=w_{01}, \quad c_1(w_0)=\tau w_{02}-w_{01}^2\\
  &\pi_2(w_0)=\begin{bmatrix}
    w_{01}\\
    w_{01}+\tau w_{02}
  \end{bmatrix}, \quad c_2(w_0)=\tau w_{02}-\tau^2w_{01}-w_{01}^2\\
  &\pi_3(w_0)=\begin{bmatrix}
    w_{01} \\ \tau w_{02}
  \end{bmatrix}, \quad c_3(w_0)=w_{01}^3+(1-\tau^2)w_{01}-\tau w_{02}
\end{align}
\end{subequations}
Also, it can be seen that the agents satisfy Assumption A2) and A3). The feedback gains are designed as
\begin{subequations}
  \begin{align}
    &K_1=-5\\
    &K_2 = [-12,-8], \quad L_2^T= [-8, -20]\\
    &K_3 = [-11,-8], \quad L_3^T= [-10,-30]
  \end{align}
\end{subequations}
For agent $1$ is a state feedback controller, while for agent $2$ and $3$ are output feedback controllers. The communication graph is switched randomly among three digraphs in Fig. \ref{fig01} with a fixed time interval $D_t=0.25$s.

Simulation results are shown in Fig. \ref{fig02} with angle frequency $\tau=10$ , where all initial conditions are randomly produced with each element being in the region of $[-1,1]$. It can be seen after transition time, all the agents have their outputs not only synchronize but also demonstrate a sinusoid wave, although the associated communication graph is randomly switching.

\begin{figure}
  \centering
  \includegraphics[width=\textwidth]{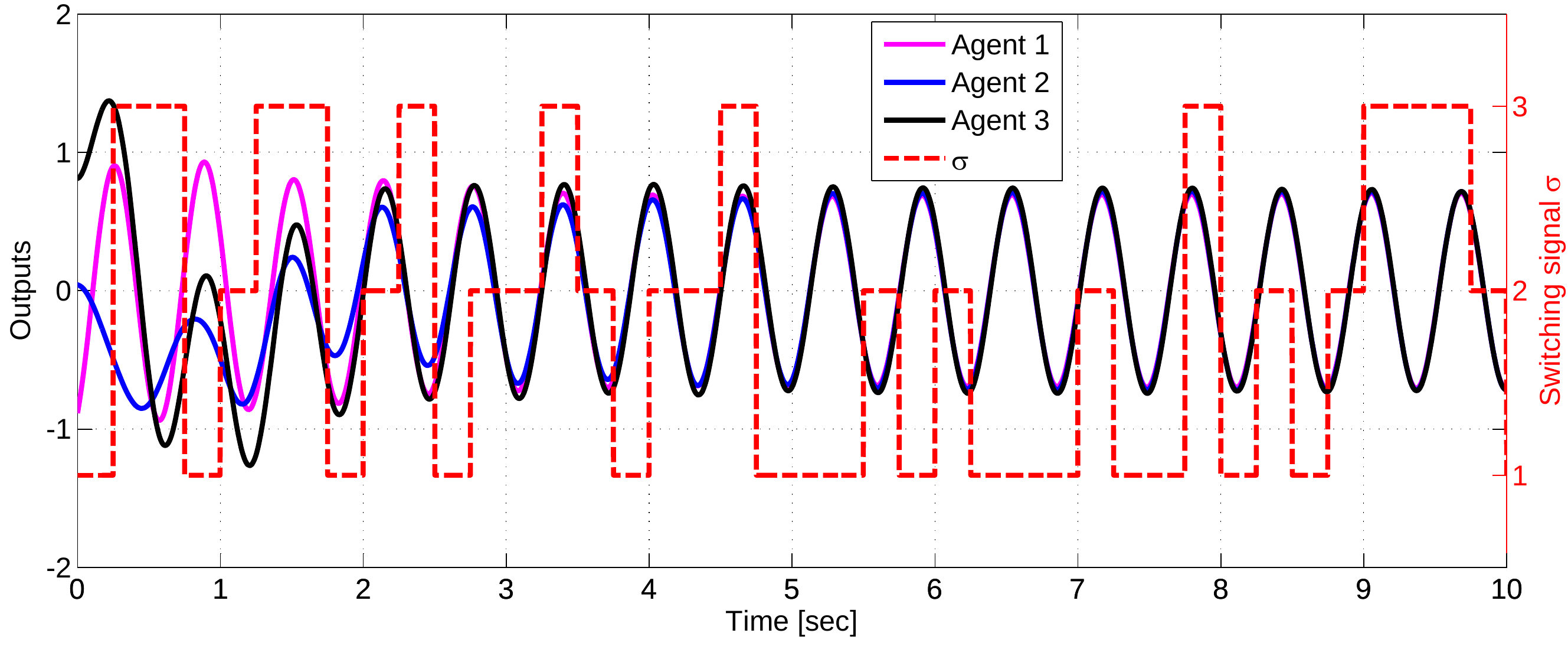}\\
  \caption{Trajectories of outputs of three agents, $y_i$, $i=1,2,3$, and the graph index $\sigma(t)$}\label{fig02}
\end{figure}

\section{Conclusion} \label{sec06}
It has been shown that the sufficient and necessary condition that the agent dynamics should satisfy for the solvability of SOR problem is the same as that for nonlinear output regulation problem. Both the dynamic state feedback controller and the dynamic output feedback controller have been respectively presented. Both of them can achieve the SOR if the switching graph satisfies the bounded interconnectivity times condition. Extension to error feedback controller is an appealing topic for future work.

\bibliographystyle{unsrt}        

\end{document}